\documentclass[11pt]{article}
\usepackage{mathrsfs}
\usepackage{amsfonts}
\usepackage[centertags]{amsmath}
\usepackage{amssymb}
\usepackage{amsthm}
\usepackage{enumerate}
\usepackage{graphicx}
\usepackage{appendix}
\usepackage{multirow}
\newtheorem{thm}{Theorem}[section]

\newtheorem{lem}[thm]{Lemma}
\newtheorem{cor}[thm]{Corollary}

\def\a{\alpha}

 \def\O{\Omega}
\def\og{\overline{G}}

\def\di{\bigm|} \def\lg{\langle} \def\rg{\rangle}
\def\nd{\mathrel{\bigm|\kern-.7em/}}

\def\Aut{\hbox{\rm Aut\,}}

\def\A{\mathcal{A}$$}

\begin{document}
\title{A classification of finite $p$-groups with a unique $\mathcal{A}_2$-subgroup\thanks{This
research was partially supported by the National Natural Science Foundation of China (nos. 11771258, 11971280)}}

\author{Jixia Gao$^{1}$, Dandan Zhang$^{1}$, Haipeng Qu$^{1,2}$\thanks{corresponding author}\\
\footnotesize\em $1$. School of Mathematics and Statistics, Lanzhou University,\\
\footnotesize\em Lanzhou Gansu 730000, P. R. China\\
\footnotesize\em e-mail: orcawhale@163.com\\
\footnotesize\em $2.$ School of Mathematics and Computer Science, Shanxi Normal University,\\
\footnotesize\em Taiyuan Shanxi 030032, P. R. China}
\date{}
\maketitle

\begin{abstract}
Finite $p$-groups with a unique $\mathcal{A}_2$-subgroup are classified  up to isomorphism. A problem proposed by Berkovich and Janko is solved.


\medskip
\noindent{\bf 2020MSC} 20D15
\end{abstract}

\baselineskip=16pt

\section{Introduction}

In this paper, $p$ is a prime, all groups considered are finite $p$-groups (in brief, $p$-groups), where $p$-groups are the groups of prime power order. A finite group $G$ is said to be minimal non-abelian if $G$ is non-abelian but all its proper subgroups are abelian. Berkovich and Janko in their long paper \cite{At} introduced a more general concept than that of minimal non-abelian $p$-groups.
A finite non-abelian $p$-group is called an $\mathcal{A}_t$-group, $t$ is a positive integer, if it has a non-abelian subgroup of index $p^{t-1}$ but all its subgroups of index $p^t$ are abelian. Obviously, $\mathcal{A}_1$-groups are exactly minimal non-abelian $p$-groups, which were classified by R${\rm\acute{e}}$dei in \cite{A1}. For $t=2,3$, $\mathcal{A}_t$-groups were classified by Zhang et al. in \cite{A2,A3}, respectively. Obviously, given a non-abelian $p$-group $G$, there is a $t\in \mathbb{N}$ such that $G$ is an $\mathcal{A}_t$-group. Hence the study of finite non-abelian $p$-groups can be regarded as that of $\mathcal{A}_t$-groups for
some $t\in \mathbb{N}$. We use $G\in \mathcal{A}_t$ to denote $G$ is an $\mathcal{A}_t$-group.

As numerous results show, the structure of a $p$-group depends essentially on its $\mathcal{A}_1$-subgroups. Indeed, the structure of a $p$-group also depends essentially on the $\mathcal{A}_2$-subgroups contained there. Recently, some scholars studied the structure of $p$-groups by imposing restrictions on their $\mathcal{A}_2$-subgroups, see \cite{BerZhang,Zhang DD,ZZ}. In this paper, we study the structure of $p$-groups by imposing restriction on the number of their $\mathcal{A}_2$-subgroups.

Following the notation of \cite{BJI}, $\alpha_k(G)$ denotes the number of $\mathcal{A}_k$-subgroups of a non-abelian $p$-group $G$.
An and Qu et al. in a series of papers \cite{Z4,Z2,Z1,Z3,Z5} classified the $p$-groups with an $\mathcal{A}_1$-subgroup of index $p$. Naturally, we wish to classify the $p$-groups with an $\mathcal{A}_2$-subgroup of index $p$. Such a class of $p$-groups contains at least $\mathcal{A}_3$-groups as its subclass. Hence such a class of $p$-groups is quite large, and this classification is an enormous work!
In this paper, we classify the $p$-groups with a unique $\mathcal{A}_2$-subgroup. In other words, we classify the $p$-groups with $\alpha_2(G)=1$. Thus \cite[Problem 895]{B2} are solved.
Due to the classification of $\mathcal{A}_3$-groups and the classification of
the $p$-groups obtained by An and Qu et al., the classification of the $p$-groups with a unique $\mathcal{A}_2$-subgroup is completed in a very short length.

\section{Preliminaries}

Let $G$ be a finite $p$-group. We use $c(G)$, $\exp(G)$ and $d(G)$ to denote the nilpotency class, the exponent and the minimal cardinality of generating set of $G$, respectively. We use ${\rm C}_{p^n}$ and ${\rm C}_{p^n}^m$ to denote the cyclic group of order $p^n$ and the direct product of $m$ copies of ${\rm C}_{p^n}$, respectively. The group $G$ is the central product,
$G = A \ast B$, of the two subgroups $A$ and $B$ if $G = AB$ and $[A, B] = 1$. 

For any positive integer $s$, we define
$$\Omega_s(G)=\langle a\in G\mid a^{p^s}=1\rangle~ {\rm and}~  \mho_s(G)=\langle a^{p^s}\mid a\in G\rangle.$$
The Frattini subgroup $\Phi(G)$ of $G$ is  equal to $G'\mho_1(G)$.

We use ${\rm M}_p(n,m)$ to denote the $p$-groups
$$\langle a,b\di a^{p^n}=b^{p^m}=1,~[a,b]=a^{p^{n-1}}\rangle,~{\rm where}~n\geqslant2,~m\geqslant1. $$

We use ${\rm M}_p(n,m,1)$ to denote the $p$-groups
$$\langle a,b,c\di a^{p^n}=b^{p^m}=c^p=1,~[a,b]=c,~[c,a]=[c,b]=1\rangle,$$
${\rm where}~n\geqslant m\geqslant1,~{\rm and~ if}~ p=2,~{\rm then}~m+n\geqslant3$.


For any undefined notations and terminologies, the reader is referred to \cite{BJI} .

\begin{lem}\label{l1}{\rm (\cite[Lemma 2.2]{Zhang2})}
Let $G$ be a finite $p$-group. Then the following propositions are equivalent$:$

$(1)$ $G$ is a minimal non-abelian $p$-group$;$

$(2)$ $d(G)=2$ and $|G'|=p;$

$(3)$ $d(G)=2$ and $Z(G)=\Phi(G).$
\end{lem}

\begin{lem}\label{l3}{\rm (\cite[Corollary 2.4]{Z2})}
{\rm (1)} Let $M$ be an $\mathcal {A}_t$-group and
$A$ be an abelian group of order $p^k$. Then $G=M\times A$ is an $\mathcal {A}_{t+k}$-group.

{\rm (2)} Let $M$ be an $\mathcal {A}_t$-group with $|M'|=p$, $G=M\ast A$, where $A$ is abelian with order $p^{k+1}$ and $M\cap A=M'$. Then $G$ is an $\mathcal {A}_{t+k}$-group.
\end{lem}

\begin{lem}\label{l14}{\rm(\cite[Lemma 2.15]{A3})} Suppose that $G$ is a finite $p$-group. $M_1$ and $M_2$ are two distinct maximal subgroups of $G$. Then
$|G'|\leqslant p|M_1'M_2'|$.
\end{lem}

\begin{lem}{\rm (\cite[Lemma 1.4]{BJI})}\label{l17} Let $G$ be a $p$-group and $N\trianglelefteq G$. If $N$ has no abelian $G$-invariant subgroups of type $(p,p)$, then $N$ is either cyclic or a $2$-group of maximal class. If, in addition, $N\le \Phi(G)$, then $N$ is cyclic.
\end{lem}

\begin{lem}{\rm (\cite[Proposition 1.13]{BJI})}\label{l18} Let $G$ be a $p$-group and let $N\le \Phi(G)$ be $G$-invariant. If $Z(N)$ is cyclic, then $N$ is also cyclic.
\end{lem}

\section{A classification of $p$-groups with a unique $\mathcal{A}_2$-subgroup}
In the following we always assume $G$ is an $\mathcal{A}_t$-group with $t\geqslant3$. We discuss in two cases: $|G'|=p$ and $|G'|\geqslant p^2$.

\begin{thm}\label{zd1}
Let $G$ be a finite $p$-group with $|G'|=p$. Then $G$ has a unique $\mathcal{A}_2$-subgroup if and only if $G=H\ast C$, where $H$ is a non-abelian subgroup of order $p^3$ and $C$ is cyclic.
\end{thm}
\begin{proof}
$(\Leftarrow)$ Since $|G'|=p$, $M'=G'$ for any non-abelian subgroup $M$ of $G$.
Let $L/G'=\Omega_1(G/G')$. Then $L/G'\cong {\rm C}_p^3$.
Let $K$ be an $\mathcal{A}_2$-subgroup of $G$. Then $d(K)\geqslant 3$  by Lemma \ref{l1}.
It follows that $|\O_1(K/K')|\geqslant p^3$.
This forces that $\Omega_1(K/G')=\Omega_1(K/K')=\O_1(G/G')$.
Hence $L/G'\leq K/G'$. Thus $L\leq K$. Notice that $H<L$. We have $H<K$.
Since $K$ is an $\mathcal{A}_2$-group, $|K/H|=p$.  This forces $K=L$. That is, $G$ has a unique $\mathcal{A}_2$-subgroup.

$(\Rightarrow)$
Suppose that $H$ is a non-abelian subgroup of minimum order of $G$. Since $|G'|=p$, $H'=G'$. In particular, $H\unlhd G$.
If $K/H$ is a subgroup of $G/H$ of order $p$, then $|K:H|=p$. By the minimality of $H$, each subgroup of $K$ of index $\geqslant p^2$
is abelian. Hence $K$ is an $\mathcal{A}_2$-group.
Since $G$ has a unique $\mathcal{A}_2$-subgroup,  $G/H$ has a unique subgroup of order $p$. Obviously, $G/H$ is abelian. Hence $G/H$ is cyclic.
By the minimality of $H$, $d(H)=2$. Thus $d(G)\leqslant 3$. By Lemma~\ref{l1}, $d(G)\ne 2$. Hence $d(G)=3$.
Fortunately, the $p$-groups $G$ with $d(G)=3$ and $|G'|=p$ were classified, which are listed in \cite[Lemma 3.1]{Z2}. The conclusion follows by a simple checking.
\end{proof}

\begin{cor}\label{zc1}
Let $G$ be a finite $p$-group with a unique $\mathcal{A}_2$-subgroup $K$ and $|G'|=p$. Then $K\not\leq \Phi(G)$.
\end{cor}

\begin{lem}\label{zl31}
Let $G$ be a finite $p$-group with a unique $\mathcal{A}_2$-subgroup $K$ and $|G'|\geqslant p^2$. Then $K\nleq \Phi(G)$ if and only if $G$ has an $\mathcal{A}_1$-subgroup of index $p$.

\end{lem}
\begin{proof}
$(\Leftarrow)$ Let $M$ be an $\mathcal{A}_1$-subgroup of index $p$ of $G$. Then $K\nleq M$. Hence $K\nleq \Phi(G)$.

$(\Rightarrow)$
Let $M$ be a maximal subgroup of $G$. If $M$ is neither abelian nor minimal non-abelian, then $K\le M$. Moreover, $K\Phi(G)\le M$.
Let $d(G)=d$ and $d(G/K\Phi(G))=k$. Since $K\nleq \Phi(G)$,  $k<d$. Notice that $G$ has $1+p+\cdots+p^{d-1}$ maximal subgroups and $G/K\Phi(G)$ has
$1+p+\cdots+p^{k-1}$ maximal subgroups. Thus $G$ has at least $p^{d-1}$ maximal subgroups which are abelian or minimal non-abelian.
It follows from $|G'|\geqslant p^2$ and Lemma \ref{l14} that $G$ has at most one abelian maximal subgroup.
Thus $G$ has at least $(p^{d-1}-1)$ $\mathcal{A}_1$-subgroup of index $p$.
\end{proof}

By the proof process $(\Rightarrow)$ of Lemma \ref{zl31}, we have

\begin{cor}\label{c1}
Assume a $p$-group $G$ has a unique $\mathcal{A}_1$-subgroup $A$ of index $p$ and $|G'|\geqslant p^2$. If $G$ has a unique $\mathcal{A}_2$-subgroup, then $G$ is a two-generator $2$-group with an abelian maximal subgroup. In particular, $G/A'$ has at least two abelian maximal subgroups.
\end{cor}

\begin{lem}\label{juzhe}
Assume $G$ is a finite $p$-group, $d(G)=3$, $\Phi(G)\leq Z(G)$ and $G'\cong {\rm C}_p^3$. Let $\overline{G}=G/G'=\lg \overline{a_1},~\overline{a_2}, ~\overline{a_3}\rg$, where $o(\overline{a_i})=p^{m_i}$, $i=1,2,3$ and $m_1\geqslant m_2\geqslant m_3$.
Let $\omega(G)=(\omega_{ij})$ be a $3\times 3$ matrix over $GF(p)$ and $a_i^{p^{m_i}}=x^{\omega_{i1}}y^{\omega_{i2}}z^{\omega_{i3}}$,  where $x=[a_2,a_3],~y=[a_3,a_1]$,~$z=[a_1,a_2]$. If $\omega(G)$ has at least two principal minors to be equal $0$, then $G$ has at least two $\mathcal{A}_2$-subgroups.
\end{lem}
\begin{proof}
Without loss of generality assume ${\scriptsize\left|
\begin{array}{cc}
\omega_{11}& \omega_{12}\\
 \omega_{21}& \omega_{22}\\ \end{array} \right|}=$
 ${\scriptsize\left|
\begin{array}{cc}
\omega_{22}& \omega_{23}\\
 \omega_{32}& \omega_{33}\\ \end{array} \right|}=0$.
Let $H=\lg a_1,a_2\rg$. Notice that $G'\le \Phi(G)\le Z(G)$.  We have $H'=\lg [a_1,a_2]\rg$. Moreover, $|H'|=p$. Then $H\in \mathcal{A}_1$ by Lemma \ref{l1}.
We claim $H\cap G'=\langle a_1^{p^{m_1}}, a_2^{p^{m_2}}, z\rangle$. In fact,
take $g\in H\cap G'$. Then $g=a_1^ia_2^jz^k$. Moreover, $\overline{1}=\overline{g}=\overline{a_1}^i\overline{a_2}^j$.
It follows from $\lg \overline{a_1},~\overline{a_2},~\overline{a_3}\rg$ being a basis of $\overline{G}$ that $\overline{a_1}^i=\overline{a_2}^j=\overline{1}$. Hence $p^{m_1}|i$ and $p^{m_2}|j$. Thus $g\in \langle a_1^{p^{m_1}}, a_2^{p^{m_2}}, z\rangle$. That means $H\cap G'\leq\langle a_1^{p^{m_1}}, a_2^{p^{m_2}}, z\rangle$. Obviously, $\langle a_1^{p^{m_1}}, a_2^{p^{m_2}}, z\rangle\leq H\cap G'$. Thus $H\cap G'=\langle a_1^{p^{m_1}}, a_2^{p^{m_2}}, z\rangle$.
Since ${\scriptsize\left|
\begin{array}{cc}
\omega_{11}& \omega_{12}\\
 \omega_{21}& \omega_{22}\\ \end{array} \right|}=0$,
 $|\langle a_1^{p^{m_1}}, a_2^{p^{m_2}}, z\rangle|\leqslant p^2$.
Since $G'\cong {\rm C}_p^3$, there exists $c\in G'\setminus H$. By Lemma \ref{l3}$(1)$, $H\times  \langle c\rangle$ is an $\mathcal{A}_2$-subgroup of $G$.
Similarly, there exists $d\in G'\setminus \langle a_2,a_3\rangle$ such that $\langle a_2,a_3\rangle\times  \langle d\rangle$ is also an $\mathcal{A}_2$-subgroup of $G$.
Clearly, $\lg a_1,a_2\rg \times \lg c\rg\ne \langle a_2,a_3\rangle\times  \langle d\rangle$. Thus $G$ has at least two $\mathcal{A}_2$-subgroups.
\end{proof}

In following we classify the $p$-groups $G$ with a unique $\mathcal{A}_2$-subgroup $K$ in two cases: $K\nleq \Phi(G)$ and $K\leq \Phi(G)$.

\medskip

If $K\nleq \Phi(G)$ and $|G'|\geqslant p^2$, then, by Lemma \ref{zl31},
$G$ has an $\mathcal{A}_1$-subgroup of index $p$. Fortunately, such $p$-groups have been classified up to isomorphism in \cite{Z4,Z2,Z1,Z3,Z5}.
The classification results are listed in \cite[Theorem 7.1]{Z4}, \cite[Theorem 4.12]{Z2}, \cite[Lemma 2.12, Corollary 3.5, Theorem 3.8, 3.9]{Z1},  \cite[Theorem 3.1, 4.1, 5.1, 6.1, 7.1, 7.6]{Z3} and \cite[Theorem 5.1, 5.3]{Z5}. In this case, what we need to do is to pick out the groups with a unique $\mathcal{A}_2$-subgroup from the groups listed in the theorems mentioned above. Although there are a lot of the $p$-groups with an $\mathcal{A}_1$-subgroup of index $p$, the process of checking can be greatly reduced by Corollary \ref{c1} and Lemma \ref{juzhe}. For the remaining cases, the process of checking is simple but tedious. The details are omitted.
Now we list the groups with $K\nleq \Phi(G)$ and $|G'|\geqslant p^2$ as follows. In particular, we also list some information of these groups, which are used in the proof of Theorem \ref{zd22}.

\begin{thm}\label{zd31}
Let $G$ be a finite $p$-group with $|G'|\geqslant p^2$. Then $G$ has a unique $\mathcal{A}_2$-subgroup $K$ and $K\nleq \Phi(G)$ if and only if $G$ is one of the following non-isomorphic groups$:$
\begin{itemize}
  \item[$(1)$] $\langle a, b, c \di a^{4}=b^{4}=c^{2}=1,~[b,c]=1,~[c,a]=b^{2},~[a,b]=a^{2} \rangle;$ moreover, $G\in \mathcal{A}_3,~|G|=2^5$, $K=\langle ab,c\rangle\times\langle a^2\rangle\cong {\rm D}_8\times {\rm C}_2$.

  \item[$(2)$] $\langle a, b, c,x \di a^{4}=b^{2}=x^2=1,~[a,b]=x,~[a,c]=a^2=c^2,~[b,c]=[x,a]=[x,b]=[x,c]=1 \rangle;$ moreover, $G\in \mathcal{A}_3,~|G|=2^5$, $K=\langle a,c\rangle\times\langle x\rangle\cong {\rm Q}_8\times {\rm C}_2$.

  \item[$(3)$] $\langle a, b, c \di a^{8}=c^{2}=1,~ b^2=a^{4},~[a,b]=c,~[c,a]=b^2,~[c,b]=1 \rangle;$ moreover,  $G\in \mathcal{A}_3,~|G|=2^5$, $K=\langle a^2,b\rangle\times\langle c\rangle\cong {\rm Q}_8\times {\rm C}_2$.

  \item[$(4)$] $\langle a, b, c \di a^{8}=c^4=1,~b^2=a^4,~[a,b]=c,~[c,a]=a^4,~[c,b]=c^2\rangle;$ moreover,  $G\in \mathcal{A}_3,~|G|=2^6$, $K=\langle a^2,b,c\rangle\cong{\rm M}_2(2,2,1).{\rm C}_2$.

  \item[$(5)$] $\langle a, b, c \di a^{2^{n+1}}=b^{2}=c^2=1,~[a,b]=c,~[c,a]=a^{2^n},~[c,b]=1 \rangle$, where $n\geqslant 2;$ moreover, $G\in \mathcal{A}_3,~|G|=2^{n+3}$, $K=\langle a^2,b\rangle\times\langle c\rangle\cong{\rm M}_2(n,1)\times {\rm C}_2$.

  \item[$(6)$] $\langle a, b, c, d \di a^{2^n}=b^{2}=c^2=d^2=1,~[a,b]=c,~[c,a]=d,~[c,b]=[d,a]=[d,b]=1 \rangle$, where $n\geqslant 2;$
 moreover,    $G\in \mathcal{A}_3,~|G|=2^{n+3}$, $K=\langle a^2,b\rangle\times\langle c\rangle\cong{\rm M}_2(n-1,1,1)\times {\rm C}_2$.

  \item[$(7)$] $\langle a, b, c \di a^{2^n}=b^{4}=c^{2}=1,~[a,b]=c,~[c,a]=b^2,~[c,b]=1 \rangle$, where $n\geqslant 3;$
moreover,   $G\in \mathcal{A}_3,~|G|=2^{n+3}$, $K=\langle b,a^2\rangle\times\langle c\rangle\cong{\rm M}_2(2,n-1)\times {\rm C}_2$.

  \item[$(8)$] $\langle a,b,c\di a^{8}=b^{4}=c^{4}=1,~[b,c]=a^{4},~[c,a]=b^2,~[a,b]=b^2c^2\rangle;$
moreover,    $G\in \mathcal{A}_3,~|G|=2^7$, $K=\langle b,c\rangle\ast\langle a^2\rangle\cong {\rm M}_2(2,2,1)\ast{\rm C}_4$.

  \item[$(9)$] $\langle a, b, c \di a^{4}=b^{4}=c^{4}=1,~[a,b]=c,~[c,a]=b^2,~[c,b]=c^2\rangle;$
 moreover,   $G\in \mathcal{A}_4,~|G|=2^6$, $K=\langle c,a^2b\rangle\times\langle b^2\rangle\cong{\rm Q}_8\times {\rm C}_2$.

  \item[$(10)$] $\langle a, b,c \di b^2=c^4=1,~a^{2^n}=c^2,~[a,b]=c,~[c,a]=[c,b]=c^2\rangle$, where $n \geqslant 2;$
 moreover,   $G\in \mathcal{A}_{n+1},~|G|=2^{n+3}$, $K=\langle c,b\rangle\ast\langle a^{2^{n-1}}\rangle\cong{\rm D}_8\ast {\rm C}_4$.

  \item[$(11)$] $\langle a, b,c\di a^{2^n}=b^2=c^4=1, ~[a,b]=c, ~[c,a]=[c,b]=c^2 \rangle$, where $n\geqslant 2;$
moreover,  $G\in \mathcal{A}_{n+1},~|G|=2^{n+3}$, $K=\langle c,b\rangle\times\langle a^{2^{n-1}}\rangle\cong{\rm D}_8\times {\rm C}_2$.

  \item[$(12)$] $\langle a, b,c\di a^{2^n}=c^4=1, ~b^2=c^2,~[a,b]=c,~[c,a]=[c,b]=c^2\rangle$, where $n\geqslant 2;$
moreover,  $G\in \mathcal{A}_{n+1},~|G|=2^{n+3}$, $K=\langle c,b\rangle\times\langle a^{2^{n-1}}\rangle\cong{\rm Q}_8\times {\rm C}_2$.

  \item[$(13)$] $\langle a, b, c \di a^{4}=b^{2^{n+1}}=c^{4}=1,~[b,c]=1,~[c,a]=a^2=c^2, ~[a,b]=b^{2^n} \rangle$, where $n\geqslant2;$
moreover,   $G\in \mathcal{A}_{n+2},~|G|=2^{n+4}$, $K=\langle a,c\rangle\times \langle b^{2^n}\rangle\cong {\rm Q}_8\times {\rm C}_2$.

  \item[$(14)$] $\langle a, b, c\di a^{2^{n+1}}=b^4=1,~c^2=b^2,~[a,b]=c,~[c,a]=a^{2^n},~[c,b]=c^2\rangle$, where $n\geqslant 3;$
  moreover,    $G\in \mathcal{A}_{n+2},~|G|=2^{n+4}$, $K=\langle c,b\rangle\times\langle a^{2^{n}}\rangle\cong {\rm Q}_8\times {\rm C}_2$.

  \item[$(15)$] $\langle a, b, c,d\di a^{9}=c^{3}=d^3=1, ~b^3=a^3,~[a,b]=c,~[c,a]=d,~[c,b]=a^{3},~[d,a]=[d,b]=1\rangle$;
  moreover,    $G\in \mathcal{A}_3,~|G|=3^5$, $K=\langle b,c\rangle\times\langle d\rangle\cong{\rm M}_3(2,1)\times {\rm C}_3$.

  \item[$(16)$] $\langle a, b, c,d \di a^{9}=b^{3}=c^3=d^3=1,~[a,b]=c,~[c,a]=d,~[c,b]=a^{-3},~[d,a]=[d,b]=1\rangle;$
moreover,    $G\in \mathcal{A}_3,~|G|=3^5$, $K=\langle c,b\rangle\times\langle d\rangle\cong{\rm M}_3(1,1,1)\times {\rm C}_3$.

  \item[$(17)$] $\langle a, b, c,d \di a^{p}=b^{p^{2}}=c^{p}=d^p=1,~[a,b]=c,~[c,a]=b^{\nu p},~[c,b]=d,~[d,a]=[d,b]=1\rangle$, where $p>3$, $\nu=1$ or a fixed quadratic non-residue modular $p;$
          moreover,    $G\in \mathcal{A}_3,~|G|=p^5$, $K=\langle a,c\rangle\times\langle d\rangle\cong{\rm M}_p(1,1,1)\times {\rm C}_p$.

  \item[$(18)$] $\langle a, b, c \di a^{p}=b^{p^2}=c^{p^2}=1,~[b,c]=1,~[c,a]=b^{p}c^{p},~[a,b]=b^{-p}\rangle$, where $p\geqslant3$;
      moreover,     $G\in \mathcal{A}_3,~|G|=p^5$, $K=\langle b,a\rangle\times\langle c^p\rangle\cong{\rm M}_p(2,1)\times {\rm C}_p$.

  \item[$(19)$] $\langle a, b, c \di a^{p^2}=b^{p^2}=c^{p}=1,~[a,b]=c,~[c,a]=a^{p}b^{\nu p},~[c,b]=b^{p}\rangle$, where $p>3, ~\nu=1$ or a fixed quadratic non-residue modular $p;$
        moreover,    $G\in \mathcal{A}_3,~|G|=p^5$, $K=\langle b,c\rangle\times\langle a^p\rangle\cong{\rm M}_p(2,1)\times {\rm C}_p$.

  \item[$(20)$] $\langle a, b, c \di a^{p}=b^{p^{3}}=c^{p^{2}}=1,~[b,c]=1,~[c,a]=b^{p^2},~ [a,b]=c^{\nu p} \rangle$, where $\nu=1$ or a fixed square non-residue modulo $p;$
        moreover,     $G\in \mathcal{A}_3,~|G|=p^6$, $K=\langle c,a\rangle\ast\langle b^p\rangle\cong{\rm M}_p(2,1,1)\ast {\rm C}_{p^2}$.

  \item[$(21)$] $\langle a, b, c \di a^{p^3}=b^{p^{2}}=c^p=1,~[a,b]=c,[c,a]=b^{\nu_1p},[c,b]=a^{-\nu_2 p^2}\rangle$, where $p\geqslant3$, $\nu_1,\nu_2=1$ or a fixed quadratic non-residue modular $p;$
   moreover,    $G\in \mathcal{A}_3,~|G|=p^6$, $K=\langle c,b\rangle\ast\langle a^p\rangle\cong{\rm M}_p(2,1,1)\ast {\rm C}_{p^2}$.

  \item[$(22)$] $\langle a, b, c \di a^{p^2}=b^{p^{2}}=c^{p^2}=1,~[b,c]=a^p,[c,a]=c^{-p},[a,b]=b^pc^{\nu p}\rangle$, where $p\geqslant3$, $\nu=1$ or a fixed quadratic non-residue modular $p$, $-\nu$ is a fixed quadratic non-residue modular $p;$
     moreover,        $G\in \mathcal{A}_3,~|G|=p^6$, $K=\langle c,a\rangle\times\langle b^p\rangle\cong{\rm M}_p(2,2)\times {\rm C}_{p}$.

  \item[$(23)$] $\langle a, b, c \di a^{p^3}=b^{p^{2}}=c^{p^2}=1,~[b,c]=a^{p^2},[c,a]=b^{p},[a,b]=c^{\nu p}\rangle$, where $p\geqslant3$, $\nu=1$ or a fixed quadratic non-residue modular $p$, $-\nu$ is a fixed quadratic non-residue modular $p;$
    moreover,        $G\in \mathcal{A}_3,~|G|=p^7$, $K=\langle b,c\rangle\ast\langle a^p\rangle\cong{\rm M}_p(2,2,1)\ast {\rm C}_{p^2}$.

  \item[$(24)$] $\langle a, b, c \di a^{p^3}=b^{p^{2}}=c^{p^2}=1,~[b,c]=a^{p^2},[c,a]^{1+r}=b^{rp}c^{-p},[a,b]^{1+r}=b^pc^{p}\rangle$, where $p\geqslant3$, $r=1,2,\cdots, p-2$ and $-r$ is a fixed quadratic non-residue modular $p;$
    moreover,        $G\in \mathcal{A}_3,~|G|=p^7$, $K=\langle b,c\rangle\ast\langle a^p\rangle\cong{\rm M}_p(2,2,1)\ast {\rm C}_{p^2}$.

  \item[$(25)$] $\langle a, b, c \di a^{p^{n+1}}=b^{p}=c^p=1,~[a,b]=c,[c,a]=1,[c,b]=a^{\nu p^n} \rangle$, where $p\geqslant3$ and $n\geqslant2$, $\nu=1$ or a fixed quadratic non-residue modular $p$;
moreover,            $G\in \mathcal{A}_{n+1},~|G|=p^{n+3}$, $K=\langle c,b\rangle\ast\langle a^{p^{n-1}}\rangle\cong{\rm M}_p(1,1,1)\ast {\rm C}_{p^2}$.

  \item[$(26)$] $\langle a, b, c \di a^{p^{n}}=b^{p^{2}}=c^p=1,~[a,b]=c,~[c,a]=1,~[c,b]=b^{p} \rangle$, where $p\geqslant3$ and $n\geqslant2;$
 moreover,     $G\in \mathcal{A}_{n+1},~|G|=p^{n+3}$, $K=\langle b,c\rangle\times\langle a^{p^{n-1}}\rangle\cong{\rm M}_p(2,1)\times {\rm C}_{p}$.

  \item[$(27)$] $\langle a, b, c,d \di a^{p^{n}}=b^{p}=c^p=d^p=1,~[a,b]=c,~[c,a]=1,~[c,b]=d, ~[d,a]=[d,b]=1 \rangle$, where $p\geqslant3$ and $n\geqslant2;$
    moreover,     $G\in \mathcal{A}_{n+1},~|G|=p^{n+3}$, $K=\langle c,b\rangle\times\langle a^{p^{n-1}}\rangle\cong{\rm M}_p(1,1,1)\times {\rm C}_{p}$.

  \item[$(28)$] $\langle a, b, c \di a^{p^2}=b^{p}=c^{p^{n+1}}=1,~[b,c]=1,~[c,a]=c^{p^n},~[a,b]=a^{p}\rangle$, where $p\geqslant3$;
moreover,     $G\in \mathcal{A}_{n+2},~|G|=p^{n+4}$, $K=\langle a,b\rangle\times\langle c^{p^{n}}\rangle\cong{\rm M}_p(2,1)\times {\rm C}_{p}$.
  \end{itemize}
\end{thm}

\begin{thm}\label{zd22}
Let $G$ be a finite $p$-group with $|G'|\geqslant p^2$. Then $G$ has a unique $\mathcal{A}_2$-subgroup $K$ and $K\leq \Phi(G)$ if and only if
$G=\langle a,b,c,d\di a^4=b^4=c^2=d^2=1,~[b,a]=c,~[c,a]=d,~[d,a]=b^2,~[c,b]=[d,b]=[d,c]=1\rangle$.
\end{thm}
\begin{proof} $(\Leftarrow)$
In fact, $G=A\rtimes_\varphi \lg a\rg$, where $A=\lg b\rg\times\lg c\rg\times\lg d\rg\cong {\rm C}_4\times {\rm C}_2^2$, $a^\varphi=\a$:
$b^\a=bc$, $c^\a=cd$, $d^\a=db^2$ and $\a\in \Aut(A)$. Now we have $|G|=2^6$ and  $\Phi(G)=\lg a^2,c\rg\times \lg d\rg\cong {\rm D}_8\times {\rm C}_2$ is an  $\A_2$-subgroup of $G$. It follows that $G$ is an $\A_4$-group and all maximal subgroups of $G$ are $\A_3$-groups.
Hence the $\A_2$-subgroups of $G$ must have order $16$. It is easy to verify $G$ satisfies the hypothesis of the theorem by  calculations.

\smallskip
$(\Rightarrow)$ We discuss in two cases: $p=2$ and $p>2$.

{\bf Case 1.} $p=2$

Using Magma \cite{magma} to check the SmallGroups Database \cite{shujuku}, we know that if $|G|<2^8$, then $G$ is exactly the group in the theorem.
Assume $|G|\geqslant2^8$. We will prove there is no group satisfying the hypothesis in the theorem  by counterexample.

Assume that $G$ is a counterexample of minimal order.
Let $M$ be a maximal subgroup of $G$. We assert that $K\nleq \Phi(M)$. If not, since $K$ is a unique $\mathcal{A}_2$-subgroup of $G$, $K$ is a unique $\mathcal{A}_2$-subgroup of $M$. By Corollary \ref{zc1}, $|M'|\geqslant 2^2$. Thus $M$ satisfies the conditions of the theorem.
By the minimality of $G$, $M$ satisfies the conclusion of the theorem. Thus $M$ is the group in the theorem. That is, $|M|=2^6$. This contradicts $|G|\geqslant2^8$. Thus the assert holds. Now we have $M$ is one of the groups in Theorem \ref{zd1} and Theorem \ref{zd31}.
By Lemma \ref{l18}, $Z(K)$ is non-cyclic. Thus $M$ is one of the groups in Theorem \ref{zd1} with $H\cap C=1$ or the groups $(5-8)$ and $(11-14)$ in Theorem \ref{zd31}.

We assert that $M$ cannot be the group $(5)$ in Theorem \ref{zd31}. If not, then $|M|=2^{n+3}$ and $K=\langle a^2,b,c\rangle\cong {\rm M}_2(n,1)\times {\rm C}_2$. It follows from $|M|\geqslant 2^7$ that $n\geqslant4$.
Since $K$ is a unique $\mathcal{A}_2$-subgroup of $G$ and $K\leq \Phi(G)$, $K$ is a unique $\mathcal{A}_2$-subgroup of any maximal subgroup $T$ of $G$.
Same as $M$, $T$ is one of the groups in Theorem \ref{zd1} with $H\cap C=1$ or the groups $(5-8)$ and $(11-14)$ in Theorem \ref{zd31}.
Now $T$ and $K$ are known. From this we know only the $\mathcal{A}_2$-subgroup of the group (5) is isomorphic to ${\rm M}_2(n,1)\times {\rm C}_2$, where $n\geqslant4$.  Thus $T$ is isomorphic to the group $(5)$. By calculation, $Z(T)$ is cyclic.
On the other hand, since $K$ is a unique $\mathcal{A}_2$-subgroup of $G$, $K\unlhd G$. By $K\le \Phi(G)$ and Lemma \ref{l17},
there exists  $N\unlhd G$ and $N\le K$ such that $N\cong {\rm C}_2\times {\rm C}_2$. Notice that $G/C_G(N)\lesssim \Aut(N)$. We have $|G:C_G(N)|\leqslant 2$. Thus there exists a maximal subgroup $M_1$ of $G$ such that $N\le M_1\le C_G(N)$.
Moreover, $N\leq Z(M_1)$ and hence $Z(M_1)$ is non-cyclic. This contradicts $Z(T)$ is cyclic for any maximal subgroup $T$ of $G$.

We assert that $M$ cannot be the groups $(6)$ and $(7)$ in Theorem \ref{zd31}. If not, then $K=\langle a^2,b,c\rangle$ and $|M:K|=2$. By calculation, $\mho_2(K)=\langle a^8\rangle$ and $|K:\mho_2(K)|=2^5$.
Let $\overline{G}=G/\mho_2(K)$. Then $|\overline{G}|=2^7$.
Since $K\in \mathcal{A}_2$, $\overline{K}\in \mathcal{A}_i(i\leqslant 2)$.
Clearly, $|\overline{K}'|=2$ and $d(\overline{K})=3$. Moreover, $\overline{K}\in \mathcal{A}_2$.
We assert that $\overline{K}$ is a unique $\mathcal{A}_2$-subgroup of $\overline{G}$.
If not, let $\overline{A}$ be an $\mathcal{A}_2$-subgroup of $\overline{G}$ and $\overline{A}\ne \overline{K}$.
It follows from $\overline{A}\in \mathcal{A}_2$ that $A\in \mathcal{A}_t(t\geqslant2)$.
Since $K$ is a unique  $\mathcal{A}_2$-subgroup of $G$, $K\leq A$. Moreover, $\overline{K}\leq \overline{A}$.
This contradicts that $\overline{A}$ and $\overline{K}$ are two different $\mathcal{A}_2$-subgroups of $\overline{G}$.
Thus $\overline{K}$ is a unique $\mathcal{A}_2$-subgroup of $\overline{G}$.
Since ${K}\le {\Phi(G)}$, $\overline{K}\le \overline{\Phi(G)}=\Phi(\overline{G})$. By Corollary \ref{zc1}, $|\overline{G}'|\geqslant p^2$. Thus $\overline{G}$ satisfies the conditions of the theorem.
By the minimality of $G$, $\overline{G}$ satisfies the conclusion of the theorem. Thus $\overline{G}$ is the group in the theorem. That is, $|\overline{G}|=2^6$. This contradicts $|\overline{G}|=2^7$.

If $M$ is one of the groups in Theorem \ref{zd1} with $H\cap C=1$ or the groups $(8)$ and $(11-14)$ in Theorem \ref{zd31}, then $Z(K)\leq Z(M)$ and $\mho_1(K)\leq Z(M)$.
$K$ is isomorphic to one of the following groups$:$
\begin{itemize}
  \item[(i)] $H\times \langle x\rangle$, where $H\cong {\rm Q_8}$ or ${\rm D}_8$, $|\langle x\rangle|=2;$
  \item[(ii)] $H\ast \langle x\rangle$, where $H\cong {\rm M}_2(2,2,1)$,~$|\langle x\rangle|=4$ and $H\cap \lg x\rg=H'$.
\end{itemize}

Same  as $M$,
$Z(K)\leq Z(L)$ and $\mho_1(K)\leq Z(L)$ for any maximal subgroup $L$ of $G$. It follows that $Z(K)\le Z(G)$ and $\mho_1(K)\leq Z(G)$.

If $K$ is isomorphic to the group (i), then, let $\og=G/\langle x\rangle$. Thus $\overline{K}\cong H$.
If $K$ is isomorphic to the group (ii), then, let $\og=G/\mho_1(H)$. Thus $\overline{K}\cong {\rm M}_p(1,1,1)\ast \langle x\rangle$. In either case, we have $Z(\overline{K})$ is cyclic and $\overline{K}$ is non-cyclic. On the other hand, since ${K}\le {\Phi(G)}$,  $\overline{K}\le \overline{\Phi(G)}=\Phi(\overline{G})$. Now it follows from Lemma \ref{l18} that $\overline{K}$ is cyclic. This is a contradiction.

To sum up, the counterexample of minimal order does not exist. The conclusion follows.

{\bf Case 2.} $p>2$

In this case, we prove that there is no group satisfying the conditions in Theorem \ref{zd22}. Let $L$ be a subgroup of minimum order of $G$ such that $K\leq \Phi(L)$. Then $K\nleq \Phi(M)$ for any maximal subgroup $M$ of $L$. Hence $M$ is one of the groups in Theorem \ref{zd1} and Theorem \ref{zd31}.
By Lemma \ref{l18}, $Z(K)$ is non-cyclic. Thus $M$ is one of the groups in Theorem \ref{zd1} with $H\cap C=1$ or the groups $(15-24)$ and $(26-28)$ in Theorem \ref{zd31}. Now we have $Z(K)\leq Z(M)$ and $\mho_1(K)\leq Z(M)$. It follows that $Z(K)\le Z(L)$ and $\mho_1(K)\leq Z(L)$. It is easy to see that $K$ is isomorphic to one of the following groups$:$
\begin{itemize}
  \item[(i)] $H\times \langle x\rangle$, where $H\cong {\rm M}_p(1,1,1)$ or ${\rm M}_p(2,1)$ and $|\langle x\rangle|=p;$
  \item[(ii)] $H\ast \langle x\rangle$, where $H\cong {\rm M}_p(2,1,1)$ or ${\rm M}_p(2,2,1)$, $|\langle x\rangle|=p^2$ and $H\cap \lg x\rg=H';$
  \item[(iii)] $\langle c,a\rangle\times\langle b^p\rangle\cong {\rm M}_p(2,2)\times {\rm C}_p$.
\end{itemize}

If $K$ is isomorphic to the group (i), then, let $\overline{L}=L/\langle x\rangle$. Thus $\overline{K}\cong H$. If $K$ is isomorphic to the group (ii),
then, let $\overline{L}=L/\mho_1(H)$. Thus $\overline{K}\cong {\rm M}_p(1,1,1)\ast \langle x\rangle$. If $K$ is isomorphic to the group (iii), then, let $\overline{L}=L/\langle a^p,b^p\rangle$. Thus $\overline{K}\cong {\rm M}_p(2,1)$. In either case, we have $Z(\overline{K})$ is cyclic and $\overline{K}$ is non-cyclic. On the other hand, since ${K}\le {\Phi(L)}$,  $\overline{K}\le \overline{\Phi(L)}=\Phi(\overline{L})$. Now it follows from Lemma \ref{l18} that $\overline{K}$ is cyclic. This is a contradiction.
\end{proof}

\end{document}